\documentclass[11pt, reqno]{amsart}
\usepackage{enumitem}
\usepackage{mathtools}
\usepackage{amssymb}
\usepackage[all]{xy}
\usepackage[pdftex,dvips]{geometry}

\mathtoolsset{centercolon,mathic}

\setenumerate{topsep=0pt,labelwidth=0pt,align=left,labelindent=0pt,
labelsep=0.275in,labelwidth=0in,leftmargin=0.275in,label=\rm(\alph*),
parsep=0pt,itemsep=3pt}

\newtheoremstyle{fact}
     {9pt}
     {9pt}
     {\slshape}
     {}
     {\bfseries}
     {}
     { }
     {\thmname{#1}\thmnumber{ #2.}\thmnote{ \rm (#3)}}

\newtheorem{thm}{Theorem}[section]
\newtheorem{prop}[thm]{Proposition}

\newtheorem{cor}[thm]{Corollary}

\newtheorem{defn}[thm]{Definition}
\newtheorem{rem}[thm]{Remark}

\newcounter{mainthms}

\newtheorem{mainthm}[mainthms]{Theorem}

\theoremstyle{definition}
\newtheorem{examples}[thm]{Examples}

\theoremstyle{fact}

\newtheorem{ftheorem}[thm]{Theorem}

\newtheorem{fproposition}[thm]{Proposition}

\newtheorem{frem}[thm]{Remark}

\renewcommand{\phi}{\varphi}

\renewcommand{\subset}{\subseteq}

\newcommand{\Z}{\mathbb{Z}}

\newcommand{\R}{\mathbb{R}}

\newcommand{\LC}{\mathcal{C}}

\newcommand{\iso}{\cong}

\newcommand{\Aut}{\operatorname{Aut}}

\newcommand{\Hom}{\operatorname{Hom}}
\newcommand{\End}{\operatorname{End}}
\newcommand{\GL}{\operatorname{GL}}
\newcommand{\Homeo}{\operatorname{Homeo}}
\newcommand{\0}{\mathbf{0}}
\newcommand{\toCO}{\xrightarrow{c.o.}}
\newcommand{\tog}{\xrightarrow{g}}
\renewcommand{\1}{\mathbf{1}}

\begin{document}

\title[Decompositions of the Automorphism Group of an LCA Group]{Decompositions of the Automorphism Group of a Locally Compact Abelian Group}
\author{Iian B. Smythe}
\email{ibs24@cornell.edu}
\address{Department of Mathematics, Cornell University, Ithaca, NY 14853-4201}
\date{October 9, 2011.}

\thanks{I would like to thank the Natural Sciences and Engineering Research Council of Canada, the University of Manitoba, and Cornell University for financial support which has enabled this research.}

\begin{abstract}
	It is well known that every locally compact abelian group $L$ can be decomposed as $L_1 \oplus \R^n$, where $L_1$ contains a compact-open subgroup. In this paper, we use this decomposition to study the topological group $\Aut(L)$ of automorphisms of $L$, equipped with the $g$-topology. We show that $\Aut(L)$ is topologically isomorphic to a matrix group with entries from $\Aut(L_1)$, $\Hom(L_1,\R^n)$, $\Hom(\R^n,L_1)$, and $\GL_n(\R)$, respectively. It is also shown that the algebraic portion of the decomposition is not specific to locally compact abelian groups, but is also true for objects with a well-behaved decomposition in an additive category with kernels.
\end{abstract}

\maketitle

\section{Introduction}

Given a collection of mathematical objects with a notion of isomorphism, it is often of interest to study the self-isomorphisms, or automorphisms, of those objects. In particular, the set of all such automorphisms is a group under composition, and there is an interplay between the structure of this group of automorphisms and the underlying object. Classical examples include permutation groups of sets, which encompasses the whole of group theory, automorphism groups of fields in the context of Galois theory, and groups of diffeomorphisms of smooth manifolds. In the setting of topological spaces, where automorphisms are self-homeomorphisms of a space $X$, it is natural to consider endowing this automorphism group $\Homeo(X)$ with a topology related to that of $X$. If $X$ is locally compact, then $\Homeo(X)$ and its subgroups can be made into topological groups, via the so-called \emph{$g$-topology}, or \emph{Birkhoff topology}, generated by the subbasis consisting of sets of the form
\[
	(C,U) = \{f\in\Homeo(X): f(C)\subset U \text{ and } f^{-1}(C)\subset U\},
\]
where $C$ is a compact subset of $X$, and $U$ an open subset of $X$ \cite{A46}. This is the coarsest refinement of the compact-open topology wherein both composition and inversion are continuous.

When $L$ is a Hausdorff locally compact group, denote by $\Aut(L)$ the group of topological automorphisms of $L$, a closed subgroup of $\Homeo(L)$, endowed with the $g$-topology. In general, $\Aut(L)$ is not locally compact, even in the case where $L$ is a compact abelian group \cite[26.18 (k)]{HR}, which has led many to study conditions under which local compactness holds.  For example, if $L$ is compact, totally disconnected, and nilpotent, then local compactness, and in fact, compactness, of $\Aut(L)$ are equivalent to all Sylow subgroups having finitely many topological generators \cite{M80}. Recent work of Caprace and Monod has shown that if $L$ is totally disconnected, compactly generated and locally finitely generated, then $\Aut(L)$ is locally compact \cite[I.6]{CM11}. It is also known that $\Aut(L)$ is a Lie group provided $L$ is connected and finite dimensional \cite{LW68}. It has been shown that automorphism groups of compact abelian groups are universal for the class of non-archimedean groups in the sense that every non-archimedean group embeds as a topological subgroup of $\Aut(K)$, for some compact abelian $K$; see \cite{M76} and \cite{MS11}.

In the case where $L$ is a locally compact abelian (LCA) group, Levin \cite{L71} gave criterion for local compactness of $\Aut(L)$, provided $L$ contained a discrete subgroup such that the quotient was compact. Levin's analysis utilizes the additional structure of LCA groups afforded to us by their duality theory, and in particular, the following canonical decomposition of such groups.

\begin{ftheorem}[{\cite[24.30]{HR}, \cite[Cor. 1]{AA78}}]\label{canondecomp}
	If $L$ is an LCA group, then $L \iso L_1\oplus\R^n$, where $L_1$ is an LCA group containing a compact-open subgroup. Further, $n$ is uniquely determined, and $L_1$ is determined up to isomorphism.
\end{ftheorem}

The main result of this paper is a structural decomposition of the automorphism group of an LCA group, using the decomposition in Theorem \ref{canondecomp}:

\begin{mainthm}\label{AutLdecomp}
	Let $L$ be an LCA group with $L = L_1\oplus \R^n$, where $L_1$ contains a compact-open subgroup. Then, as topological groups,
	\[
		\Aut(L) \iso \begin{pmatrix}
					\Aut(L_1) & \Hom(\R^n,L_1)\\
					\Hom(L_1,\R^n) & \GL_n(\R)
				  \end{pmatrix},
	\]
	where the latter is equipped with the product topology.
\end{mainthm}

The algebraic portion of Theorem \ref{AutLdecomp} can be extracted and established in a more general setting.

\begin{mainthm}\label{catdecomp}
	Let $\LC$ be an additive category with kernels, and $A = B\oplus C$ an object in $\LC$ such that:
	\begin{enumerate}[label={\rm(\Roman*)}]
		\item $\delta\in\End(C)$ is an automorphism of $C$ if and only if 
the zero morphism $\0$ is a kernel of $\delta$; and
		\item For every pair of morphisms $\gamma: B \to C$ and $\beta: C \to B$, one has that $\gamma\beta=\0$.
	\end{enumerate}
Then, as groups,
	\[
		\Aut(A) \iso \begin{pmatrix}
					\Aut(B) & \LC(C,B)\\
					\LC(B,C) & \Aut(C)
				  \end{pmatrix}.
	\]
\end{mainthm}

The paper is structured as follows: In \S 2, we provide topological preliminaries regarding the compact-open and $g$-topologies. \S 3 is a discussion of an abstract categoral setting wherein we prove Theorem \ref{catdecomp}. In \S 4, we present the proof of Theorem \ref{AutLdecomp}.

\section{Preliminaries}

Throughout this paper, all spaces are assumed to be Hausdorff, and in particular, all topological groups are Tychonoff \cite[1.21]{L09}. Recall that if $X$ and $Y$ are topological spaces and $\mathcal{F}$ a collection of continuous functions from $X$ to $Y$, the \emph{compact-open topology} on $\mathcal{F}$ is the topology generated by the subbasis consisting of sets of the form
\[
	[C,U] = \{f\in\mathcal{F}: f(C)\subset U\},
\]
where $C$ is a compact subset of $X$, and $U$ an open subset of $Y$ (see \cite{F45}, \cite[\S 43]{W70}). For locally compact $X$, composition of maps is continuous in $\Homeo(X)$ when endowed with the compact-open topology, a consequence of the following property:

\begin{ftheorem}[{\cite[3.4.2]{E89}}]\label{cocomp}
	If $X$, $Y$ and $Z$ are topological spaces, with $Y$ locally compact, then the composition map $C(Y,Z) \times C(X,Y) \to C(X,Z)$ is continuous with respect to the compact-open topology.
\end{ftheorem}

However, inversion may fail to be continuous in $\Homeo(X)$ with respect to the compact-open topology \cite[p. 57-58]{B48}; this shortcoming is remedied by the $g$-topology. The two topologies coincide when $X$ is compact, discrete, or locally connected, but not in general \cite{A46}. One can characterize convergence in the $g$-topology in terms of the compact-open topology as in the following proposition.

\begin{fproposition}[{\cite[5. (ii)]{A46}}]\label{gconv}
	Let $X$ be a locally compact space. A net $(f_\lambda)$ in $\Homeo(X)$ converges to $f \in \Homeo(X)$ in the g-topology, written $f_\lambda \tog f$, if and only if $(f_\lambda)$ converges to $f$ and $(f_\lambda^{-1})$ converges to $f^{-1}$ in the compact-open topology, written $f_\lambda \toCO f$ and $f_\lambda^{-1} \toCO f^{-1}$.
\end{fproposition}

Given an LCA group $L$, $\Aut(L)$ is a closed subgroup of $\Homeo(L)$, endowed with the $g$-topology. Theorem \ref{canondecomp} implies a decomposition of $\End(L)$, the (additive) group of topological endomorphisms of $L$, endowed with the compact-open topology, into a topological ring of $2\times 2$ matrices. In particular, every element of $\Aut(L)$ can be algebraically represented in this way, but we caution that since $\Aut(L)$ carries the $g$-topology, it is \emph{not} a subspace of $\End(L)$. We note for future reference that if $L = \R^n$, then its ring of endomorphisms and group of automorphisms are familiar objects:

\begin{frem}[{\cite[26.18 (i)]{HR}}]\label{endRn}\mbox{}
	\begin{enumerate}
		\item Taken with the compact-open topology, $\End(\R^n) = M_n(\R)$, where the latter carries its standard topology as a subspace of $\R^{n^2}$.
		\item Taken with the $g$-topology, $\Aut(R^n) = \GL_n(\R)$, where the latter carries its standard topology. In particular, the compact-open and $g$-topologies on $\Aut(\R^n)$ coincide.
	\end{enumerate}
\end{frem}

\section{A Categorical Setting}

\label{sect:category}

In this section, we prove Theorem \ref{catdecomp}. First, we recall the following terminology from category theory:

\begin{defn}[{\cite[VIII.2]{M98}}]\label{catdefs}
	Let $\LC$ be a category.
	\begin{enumerate}

\item 
An object $\0$ in $\LC$ is a \emph{zero object} if for every
object $A$ of $\LC$, there are unique morphisms $\0\to A$ and $A\to\0$.
	
\item 
If $\LC$ has a zero object and $A$ and $B$ are objects in $\LC$, then
the \emph{zero morphism} $\0:A\to B$ is the 
composite of the morphism $A\to\0$ and $\0\to B$.
	
\item 
A \emph{kernel} of a morphism $f:A\to B$ is a morphism $k:K\to A$ such 
that:

\begin{enumerate}[label={\rm(\roman*)}]

\item
$fk = \0$; and

\item
every morphism $h\colon C\to A$ such that $fh=\0$ factors uniquely through 
$k$, that is, there is a unique morphism $k^\prime\colon C \to A$ 
making the following diagram commutative:

		
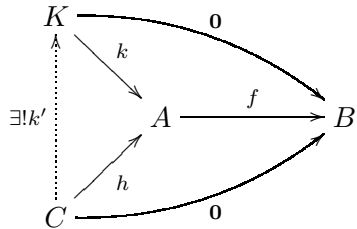
\[
			\xymatrix{K \ar[dr]^k \ar@/^1pc/[drrr]^\0 & & &\\
					& A \ar[rr]^f & & B\\
					C \ar@{.>}[uu]^{\exists ! k'} \ar[ur]_h \ar@/_1pc/[urrr]_\0 & & &}
		\]

\end{enumerate}
		
\item $\LC$ is  \emph{preadditive} if for every pair of objects 
$A$ and $B$ in $\LC$, the set $\LC(A,B)$ of morphisms from $A$ to $B$ is 
an abelian group, and the composition
$\circ\colon \LC(B,C)\times \LC (A,B) \to \LC(A,C)$ is bilinear for every 
$A$, $B$, and $C$.

\item 
$\LC$ is \emph{additive} if it has a zero object, and every two 
objects in $\LC$ have a biproduct.

\end{enumerate}
\end{defn}

\begin{examples}\mbox{ }\label{adcatex}

\begin{enumerate}

\item 
The category $\mathsf{Ab}$ of abelian groups and their homomorphisms is 
additive, with the zero object being the trivial group, and biproducts 
being direct sums.

\item 
The category $\mathsf{LCA}$ of locally compact abelian groups and their 
continuous homomorphisms is additive, with the zero object being the 
trivial group, and biproducts being direct products with the product 
topology.

\end{enumerate}
\end{examples}

In an additive category $\LC$, the abelian group
$\End(A):=\LC(A,A)$ is a ring with respect to composition for every 
object $A$ in $\LC$.

\begin{fproposition}[{\cite[p.~192]{M98}}]\label{endmat}
Let $\LC$ be an additive category and $A = A_1 \oplus A_2 \oplus \cdots 
\oplus A_n$ an object of $\LC$. Then,

	\[
		\End(A) \iso \begin{pmatrix}
							\End(A_1) & \LC(A_2,A_1) & \cdots & \LC(A_n,A_1)\\							 				\LC(A_1,A_2) & \End(A_2) & \cdots & \LC(A_n,A_2)\\
					 		\vdots & \vdots &  & \vdots\\
					 		\LC(A_1,A_n) & \LC(A_2,A_n) & \cdots & \End(A_n)\\
						 \end{pmatrix}
	\]
as rings, where composition is given by matrix multiplication.
\end{fproposition}
	
For an object $A$ in $\LC$, we denote the set of all automorphisms 
(self-isomorphisms) of $A$ by $\Aut(A)$; it is a group under composition. 

\begin{rem}\label{easy}
Let $\LC$ be an additive category, and suppose that  $A = B\oplus C$ 
is an object of  $\LC$ such that $\LC(B,C) = \{\0\}$. Then,
	\[
		\Aut(A) \iso \begin{pmatrix}
					\Aut(B) & \LC(C,B)\\
					\0 & \Aut(C)
			            \end{pmatrix}
	\]
	as groups.
\end{rem}

Theorem \ref{catdecomp} is an analogue of the aforementioned decomposition of 
$\Aut(A)$ when $\LC(B,C)$ is not necessarily trivial. To this end, for the 
remainder of this section, we fix an additive category $\LC$ such that 
every morphism has a kernel, and an object $A$ in $\LC$ such that $A = 
B\oplus C$, with $B$ and $C$ objects of $\LC$ satisfying the following 
conditions:

\begin{enumerate}[label=(\Roman{*})]

\item\label{I} $\delta\in\End(C)$ is an automorphism of $C$ if and only if 
the zero morphism $\0$ is a kernel of $\delta$.
	
\item\label{II} 
For every pair of morphisms $\gamma: B \to C$ and $\beta: C \to B$, one has
that $\gamma\beta=\0$.
\end{enumerate}

\medskip

Theorem~\ref{catdecomp} is a consequence of Proposition~\ref{endmat}, and the equivalence of (i) and (iii) in Theorem~\ref{catthm} below.

\begin{thm}\label{catthm}
	Let $\phi \in \End(A)$, where
	\[
		\phi = \begin{pmatrix}
			\alpha & \beta\\
			\gamma & \delta
		\end{pmatrix} \in
		\begin{pmatrix}
			\End(B) & \LC(C,B)\\
			\LC(B,C) & \End(C)
		\end{pmatrix}.
	\]
	Then, the following statements are equivalent:
	\begin{enumerate}[label=\rm({\roman*})]
		\item $\phi$ is an automorphism of $A$;
		
\item $\delta$ is an automorphism of $C$, and the \emph{quasi-determinant} 
$\det(\phi) := \alpha-\beta\delta^{-1}\gamma$  is an automorphism 
of $B$;
		
		\item $\delta$ is an automorphism of $C$, and $\alpha$ is an automorphism of $B$.
	\end{enumerate}
\end{thm}

We rely on the following elementary fact from ring theory in the proof of 
Theorem~\ref{catthm}.

\begin{frem}\label{1plusnil}
	Let $R$ be a (unital) ring, and $n\in R$ a nilpotent element such that $n^2 = 0$. Then, $(\1+n)^{-1} = (\1-n)$, and in particular, $(\1+n)$ is invertible.
\end{frem}

\begin{proof}{}
(i)$\Longrightarrow$(ii): In order to show that $\delta$ is automorphism, 
let $k: K\to C$ be a kernel of $\delta$.  Denote the canonical projections $\pi_1:B\oplus C\to B$ and 
$\pi_2:B\oplus C\to C$, $\iota_K\colon K\to 
B\oplus K$ and $\iota_B\colon B\to B\oplus C$ canonical injections, and 
set $\psi = \0\oplus k : B\oplus K\to B\oplus C$. Then, one has $\psi\iota_K = (0,k)$ written componentwise as a morphism into $B\oplus C$, and so $\phi\psi\iota_K = (\beta k, 0)$. Thus,
\begin{align*}
		\pi_1\phi\psi\iota_K &= \beta k.\\
\intertext{Put $g:=\pi_2\phi^{-1}\iota_B:B\to C$. Since}
\iota_B\beta k &= (\beta k, 0) = \phi\psi\iota_K, 
\intertext{one obtains that}
g\beta k = \pi_2\phi^{-1}\iota_B\beta k & = 
\pi_2\phi^{-1}\phi\psi\iota_K = \pi_2\psi\iota_K = k.
\end{align*}
However, $g\colon B\to C$ and $\beta\colon C \to B$, and so by condition 
\ref{II}, $g\beta = \0$. Therefore, $k = \0$, and it follows from 
condition \ref{I} that $\delta$ is an automorphism.

To establish the second condition, observe that $\phi$ can be expressed as 
follows:
	\begin{align}\label{elemmat}
		\phi = \begin{pmatrix}
					\1_B & \0\\
					\0 & \delta
				\end{pmatrix}
				\begin{pmatrix}
					\1_B & \beta\\
					\0 & \1_C
				\end{pmatrix}
				\begin{pmatrix}
					\det(\phi) & \0\\
					\0 & \1_C
				\end{pmatrix}
				\begin{pmatrix}
					\1_B & \0\\
					\delta^{-1}\gamma & \1_C
				\end{pmatrix}.
	\end{align}
	Since the matrices $\begin{pmatrix}
					\1_B & \0\\
					\0 & \delta
				\end{pmatrix}$, $\begin{pmatrix}
					\1_B & \beta\\
					\0 & \1_C
				\end{pmatrix}$ and $\begin{pmatrix}
					\1_B & \0\\
					\delta^{-1}\gamma & \1_C
				\end{pmatrix}$ are invertible, the remaining matrix is also invertible. The latter occurs if and only if $\det(\phi)\in\Aut(B)$.
	
	(ii)$\Longrightarrow$(i) is an immediate consequence of (\ref{elemmat}).
	
(ii)$\Longrightarrow$(iii): It is given that $\delta$ is an automorphism 
of $C$. It follows from the definition of $\det(\phi)$ that 
$\alpha = \det(\phi) + \beta\delta^{-1}\gamma$. By multiplying both sides 
by $\det(\phi)^{-1}$, one obtains
	\[
		\alpha\det(\phi)^{-1} = \1_{B} + \beta\delta^{-1}\gamma\det(\phi)^{-1}.
	\]
By condition \ref{II}, $\gamma\det(\phi)^{-1}\beta = \0$, because 
$\det(\phi)^{-1}\beta \in \LC(C,B)$, and thus
$(\beta\delta^{-1}\gamma\det(\phi)^{-1})^2 = \0$.
Therefore, by Remark~\ref{1plusnil}, $\alpha\det(\phi)^{-1}$ is 
invertible,
and its inverse is $\1_{B} - \beta\delta^{-1}\gamma\det(\phi)^{-1}$.
Hence, $\alpha$ is invertible, and 
\begin{align}\label{alphainv}
\alpha^{-1} = 
\det(\phi)^{-1}(\1_{B}-\beta\delta^{-1}\gamma\det(\phi)^{-1}).
\end{align}

(iii)$\Longrightarrow$(ii): It is given that $\delta$ is an automorphism 
of $C$. One can express 
$\det(\phi)\alpha^{-1}$ as
	\[
		\det(\phi)\alpha^{-1} = \1_B-\beta\delta^{-1}\gamma\alpha^{-1}.
	\]
	By condition \ref{II}, $\gamma\alpha^{-1}\beta = \0$, because $\alpha^{-1}\beta\in\LC(C,B)$, and thus $(-\beta\delta^{-1}\gamma\alpha^{-1})^2 = 0$. Therefore, by Remark~\ref{1plusnil}, $\det(\phi)\alpha^{-1}$ is invertible, and its inverse is $\1_B+\beta\delta^{-1}\gamma\alpha^{-1}$. Hence, $\det(\phi)$ is invertible, and
	\begin{align}\label{detinv}
		\det(\phi)^{-1} = \alpha^{-1}(\1_{B} + \beta\delta^{-1}\gamma\alpha^{-1}).
	\end{align}
This completes the proof. 
\end{proof}

The proof of Theorem \ref{catthm} also enables us to provide an explicit formula for the inverse of an element in $\Aut(A)$.

\begin{cor}\label{autinverse}
	If $\phi\in\Aut(A)$ with $\phi = \begin{pmatrix}\alpha&\beta\\\delta&\gamma\end{pmatrix}$, then
		\[
			\phi^{-1} = \begin{pmatrix}
						(\det(\phi))^{-1} &  -(\det(\phi))^{-1}(\beta\delta^{-1})\\
						-\delta^{-1}\gamma(\det(\phi))^{-1} & \delta^{-1}
					\end{pmatrix}.
		\]
\end{cor}

\begin{proof}
	The inverse $\varphi^{-1}$  can be obtained by 
expressing $\phi$ in the form provided in (\ref{elemmat}), and calculating 
the inverse of each of the factors as follows:
		\[
			\phi^{-1} =	 \begin{pmatrix}
						\1 & \0\\
						-\delta^{-1}\gamma&\1_C
				 	\end{pmatrix}
				 	\begin{pmatrix}
				 		\det(\phi)^{-1}&\0\\
						\0&\1_C
					 \end{pmatrix}
					 \begin{pmatrix}
					 	\1_B & -\beta\\
						\0 & \1_C
					 \end{pmatrix}
				 	\begin{pmatrix}
					 	\1_B & \0\\
						\0 & \delta^{-1}
					 \end{pmatrix}.
		\]
Therefore,
		\[
			\phi^{-1} = \begin{pmatrix}
						\det(\phi)^{-1} & -(\det(\phi))^{-1}(\beta\delta^{-1})\\
						-\delta^{-1}\gamma(\det(\phi))^{-1} & \delta^{-1}\gamma(\det(\phi))^{-1}\beta\delta^{-1}+\delta^{-1}
					 \end{pmatrix}.
		\]
However, $(\det(\phi))^{-1}\beta\in\LC(C,B)$, so $\gamma(\det(\phi))^{-1}\beta = \0$ by condition \ref{II}, and $\delta^{-1}\gamma(\det(\phi))^{-1}\beta\delta^{-1} = \0$.
\end{proof}

We now apply these general results to the category $\mathsf{LCA}$.
	
\begin{prop}
	$\mathsf{LCA}$ is an additive category with kernels, and the decomposition of an LCA group given in Theorem \ref{canondecomp} satisfies conditions \ref{I} and \ref{II}. That is, given $L = L_1\oplus\R^n$, where $L_1$ contains a compact-open subgroup, then:
	\begin{enumerate}
		\item $\delta\in\End(\R^n)$ is an automorphism if and only if it has trivial kernel;
		
		\item for every pair of continuous homomorphisms $\gamma:L_1\to \R^n$ and $\beta:\R^n\to L_1$, one has $\gamma\beta = \0$.
	\end{enumerate}
\end{prop}

\begin{proof}
By Example~\ref{adcatex}(b), $\mathsf{LCA}$ is an additive category. 
If $f\colon L\to H$ is a continuous homomorphism of LCA  groups, then 
the inclusion map $k\colon\ker f\to L$ is a kernel of $f$ in the sense of Definition \ref{catdefs} (c).

(a) follows from Proposition \ref{endRn}.
		
(b) Since $\R^n$ is connected, $\beta(\R^n)$ is contained in the connected 
component $c(L_1)$ of $L_1$, which is compact. One has
$\gamma(c(L_1)) = \{0\}$, because the only compact subgroup 
of $\R^n$  is the trivial one. Hence, $\gamma\beta = \0$.
\end{proof}

\begin{cor}\label{lcadet}
	Let $L = L_1\oplus\R^n$ be an LCA group, where $L_1$ contains a compact-open subgroup, and $\phi \in \End(L)$, with
	\[
		\phi = \begin{pmatrix}
			\alpha & \beta\\
			\gamma & \delta
		\end{pmatrix} \in
		\begin{pmatrix}
			\End(L_1) & \Hom(\R^n,L_1)\\
			\Hom(L_1,\R^n) & M_n(\R)
		\end{pmatrix}.
	\]
	Then, the following statements are equivalent:
	\begin{enumerate}[label=\rm({\roman*})]
		\item $\phi$ is an automorphism of $L$;
		
\item $\delta$ is an automorphism of $\R^n$ (i.e., in $\GL_n(\R)$), and the 
quasi-determinant of $\phi$ is an automorphism of $L_1$;
		
\item $\delta$ is an automorphism of $\R^n$, and $\alpha$ is an 
automorphism of $L_1$.
\hfill\qed
\end{enumerate}
\end{cor}

\section{$\Aut(L)$ and Decompositions of $L$}\label{sect:AutLdecomp}

	In this section, whenever $L$ and $H$ are LCA groups, the group $\Hom(L,H)$ of continuous homomorphisms from $L$ to $H$, and the ring $\End(L)$ of continuous endomorphisms of $L$, are endowed with the compact-open topology, while the group $\Aut(L)$ of topological automorphisms, will have the $g$-topology. We show that the results of \S\ref{sect:category} remain true for LCA groups with a topological enrichment in the sense that the algebraic isomorphisms from \S\ref{sect:category} become 
topological isomorphisms in the presence of the aforementioned topologies. The culmination of this work is Theorem \ref{AutLdecomp}, a topological enrichment of Theorem \ref{catdecomp}. We begin with the following enrichment of Proposition~\ref{endmat}:
	
\begin{prop}\label{endmattop} 
Let $L = L_1 \oplus L_2 \oplus\cdots \oplus L_n$, where each $L_i$ is an 
LCA group. Then,
	\[
		\End(L) \iso \begin{pmatrix}
							\End(L_1) & \Hom(L_2,L_1) & \cdots & \Hom(L_n,L_1)\\							 				\Hom(L_1,L_2) & \End(L_2) & \cdots & \Hom(L_n,L_2)\\
					 		\vdots & \vdots &  & \vdots\\
					 		\Hom(L_1,L_n) & \Hom(L_2,L_n) & \cdots & \End(L_n)\\
						 \end{pmatrix}
	\]
	as topological rings, where the latter is equipped with the product topology.
\end{prop}

\begin{proof}
	Let $[A_{(i,j)}]$ denote the matrix ring on the right-hand side, where $A_{(i,j)} = \End(L_i)$ if $i = j$, and $A_{(i,j)} = \Hom(L_j,L_i)$ otherwise. We define the map $F: \End(L)\to[A_{(i,j)}]$ by
	\[
		F(\phi) = [\pi_i \phi \iota_j],
	\]
	where $\pi_i$ is the canonical projection of $L$ onto $L_i$, and $\iota_j$ the inclusion of $L_j$ into $L$. By Proposition~\ref{endmat}, $F$ is a ring homomorphism from $\End(L)$ onto $[A_{(i,j)}]$. This map is continuous, since all of the spaces involved are given the compact-open topology, and the map $\phi\mapsto\pi_i\phi\iota_j$ is continuous by Proposition \ref{cocomp}. The inverse of $F$ is given by
	\[
		F^{-1}([\alpha_{i,j}]) = \sum_{(i,j)}{\iota_i \alpha_{i,j}\pi_j},
	\]
which is continuous by Proposition \ref{cocomp} and the continuity of addition in $\End(L)$.
\end{proof}
	
From now on, if $L$ is a direct sum of (finitely many) LCA groups,  we 
identify $\End(L)$ with the aforesaid matrix decomposition. 
Recall that  $\Aut(L)$ is equipped with the $g$-topology, which need not 
coincide with the topology inherited from $\End(L)$. Therefore, 
decomposition results concerning  $\End(L)$ do not automatically give rise 
to those for $\Aut(L)$. Nevertheless, in the simplest case,
a topological enrichment of Remark~\ref{easy} holds.
	
\begin{prop}\label{easytop}
	Suppose that $L = A \oplus B$, where $A$ and $B$ are LCA groups with $\Hom(A,B) = \0$. Then,
	\[
		\Aut(L) \iso \begin{pmatrix}
					\Aut(A) & \Hom(B,A)\\
					\0 & \Aut(B)
			            \end{pmatrix},
	\]
as topological groups, where the right-hand side is equipped with  the 
product topology.
\end{prop}

\begin{proof}
Define
	\[
		F\colon\Aut(L) \to \begin{pmatrix}\Aut(A) & \Hom(B,A)\\\0 & 
                \Aut(B)\end{pmatrix}, 
                \text{ by } F(\phi) = \begin{pmatrix} \alpha & \beta\\ 0 & 
                \delta\end{pmatrix}.
	\]
By Remark~\ref{easy}, $F$ is well-defined and it is a group isomomorphism, 
and in particular, $\alpha$ and $\delta$ are automorphisms of $A$ and $B$, 
respectively. Thus, it remains to be seen that $F$ is also a 
homeomorphism. Let $(\phi_\lambda)$ be a net converging (in the $g$-topology) to $\phi\in\Aut(L)$, where
	\[
		F(\phi_\lambda) = \begin{pmatrix} \alpha_\lambda & \beta_\lambda\\ 0 & \delta_\lambda\end{pmatrix}, \text{ and } F(\phi) = \begin{pmatrix} \alpha & \beta\\ 0 & \delta\end{pmatrix}.
	\]
	One may show by direct computation that
	\[
		\begin{pmatrix} \alpha_\lambda & \beta_\lambda\\ 0 & \delta_\lambda\end{pmatrix}^{-1} = \begin{pmatrix} \alpha_\lambda^{-1} & -\beta_\lambda\\ 0 & \delta_\lambda^{-1}\end{pmatrix}, \text{ and } \begin{pmatrix} \alpha & \beta\\ 0 & \delta\end{pmatrix}^{-1} = \begin{pmatrix} \alpha^{-1} & -\beta\\ 0 & \delta^{-1}\end{pmatrix}.
	\]
	Since $\phi_\lambda\tog\phi$, by Proposition \ref{gconv},
	\[
		\begin{pmatrix} \alpha_\lambda & \beta_\lambda\\ 0 & \delta_\lambda\end{pmatrix} \toCO \begin{pmatrix} \alpha & \beta\\ 0 & \delta\end{pmatrix}, \text{ and }
		\begin{pmatrix} \alpha_\lambda^{-1} & -\beta_\lambda\\ 0 & \delta_\lambda^{-1}\end{pmatrix} \toCO \begin{pmatrix} \alpha^{-1} & -\beta\\ 0 & \delta^{-1}\end{pmatrix}.
	\]
	In particular, $\alpha_\lambda\toCO\alpha$ and $\alpha_\lambda^{-1}\toCO\alpha^{-1}$, and by Proposition~\ref{gconv} applied to $\Aut(A)$, we have that $\alpha_\lambda\tog\alpha$. Similarly, $\delta_\lambda\tog\delta$ in $\Aut(C)$. It is clear that $\beta_\lambda\to\beta$ in $\Hom(B,A)$, and thus, $F(\phi_\lambda) \to F(\phi)$. Therefore $F$ is continuous.
	
	To see that $F^{-1}$ is continuous, suppose that
	\[
		\begin{pmatrix} \alpha_\lambda & \beta_\lambda\\ 0 & \delta_\lambda\end{pmatrix} \to\begin{pmatrix} \alpha & \beta\\ 0 & \delta\end{pmatrix} \text{ in } \begin{pmatrix}\Aut(A) & \Hom(B,A)\\\0 & \Aut(B)\end{pmatrix}.
	\]
	 One has that $\beta_\lambda\to\beta$ in $\Hom(B,A)$, $\alpha_\lambda\tog\alpha$ in $\Aut(A)$, and $\delta_\lambda\tog\delta$ in $\Aut(B)$. By Proposition \ref{gconv}, $\alpha_\lambda\toCO\alpha$ and $\alpha_\lambda^{-1}\toCO\alpha$, and $\delta_\lambda\toCO\delta$ and $\delta_\lambda^{-1}\toCO\delta^{-1}$. The compact-open topology on $\End(L)$ coincides with the product topology where each of the component spaces have the compact-open topology, as given in Proposition~\ref{endmattop}. Therefore,
	\[
		\begin{pmatrix} \alpha_\lambda & \beta_\lambda\\ 0 & \delta_\lambda\end{pmatrix} \toCO \begin{pmatrix} \alpha & \beta\\ 0 & \delta\end{pmatrix}, \text{ and }
		\begin{pmatrix} \alpha_\lambda^{-1} & -\beta_\lambda\\ 0 & \delta_\lambda^{-1}\end{pmatrix} \toCO \begin{pmatrix} \alpha^{-1} & -\beta\\ 0 & \delta^{-1}\end{pmatrix}.	
	\]
	Hence, $F$ is a topological isomorphism.
\end{proof}

If $L$ is a compactly generated LCA group, then $L \iso K\oplus\R^n\oplus\Z^m$, where $K$ is the maximal compact subgroup of $L$ \cite[9.8]{HR}, while if $L$ is a connected LCA group, then $L \iso K\oplus\R^n$ where $K$ is the maximal compact connected subgroup of $L$ \cite[9.14]{HR}. Combining these facts with Proposition~\ref{easytop}, we have the following:

\begin{cor}\mbox{}
	\begin{enumerate}
	\item	Let $L \iso K\oplus\R^n\oplus\Z^m$ be a compactly generated LCA group, where $K$ is the maximal compact subgroup of $L$. Then,
	\[
		\End(L) \iso \begin{pmatrix}
										\End(K) & \Hom(\R^n,K) & K^m \\
										\0 & M_n(\R) & \R^{mn} \\
										\0 & \0 & M_m(\Z)\\
								 \end{pmatrix}
		\text{ and }
		\Aut(L) \iso \begin{pmatrix}
										\Aut(K) & \Hom(\R^n,K) & K^m \\
										\0 & \GL_n(\R) & \R^{mn} \\
										\0 & \0 & \GL_m(\Z)\\
								 \end{pmatrix}.
	\]
	
	\item Let $L \iso K\oplus\R^n$ be a connected LCA group, where $K$ is the maximal compact connected subgroup of $L$. Then,
	\[
		\End(L) \iso \begin{pmatrix}\End(K) & \Hom(\R^n,K)\\\0 & M_n(\R)\end{pmatrix}, \text{ and }
		\Aut(L) \iso \begin{pmatrix}\Aut(K) & \Hom(\R^n,K)\\\0 & \GL_n(\R)\end{pmatrix}.
	\]
	\end{enumerate}
\end{cor}

\begin{proof}
	(a) The only compact subgroup of $\R^n\oplus\Z^m$ is the trivial one, and so $\Hom(K,\R^n\oplus\Z^m) = \0$. Thus, by Proposition~\ref{easytop}, we have that
	\[
		\Aut(L) \iso \begin{pmatrix}
			\Aut(K) & \Hom(\R^n\oplus\Z^m,K)\\
			\0 & \Aut(\R^n\oplus\Z^m)
		\end{pmatrix}.
	\]
	The only connected subgroup of $\Z^m$ is trivial, so $\Hom(\R^n,\Z^m) = \0$, and so
	\[
		\Aut(\R^n\oplus\Z^m) = \begin{pmatrix}
							\Aut(\R^n) & \Hom(\Z^m,\R^n)\\
							\0 & \Aut(\Z^m)
						    \end{pmatrix}.
	\]
	One can easily show that $\Hom(\R^n\oplus\Z^m,K) \iso \Hom(\R^n,K)\times\Hom(\Z^m,K)$. $\Aut(\R^n) = \GL_n(\R)$ by Remark \ref{endRn}, $\Aut(\Z^m) = \GL_m(\Z)$ by \cite[26.18(g)]{HR}, and it is elementary that $\Hom(\Z^m,G) \iso G^m$ for any topological group $G$. Hence,
	\[
		\Aut(L) \iso \begin{pmatrix}
					\Aut(K) & \Hom(\R^n,K) & \Hom(\Z^m,K) \\
					\0 & \Aut(\R^n) & \Hom(\Z^m,\R^n) \\
					\0 & \0 & \Aut(\Z)\\
				 \end{pmatrix} \iso 
				 \begin{pmatrix}
					\Aut(K) & \Hom(\R^n,K) & K^m \\
					\0 & \GL_n(\R) & \R^{mn} \\
					\0 & \0 & \GL_m(\Z)\\
				 \end{pmatrix}.
	\]
	
	(b) The only compact subgroup of $\R^n$ is the trivial one, and so $\Hom(K,\R^n) = \0$. The remainder of the result follows from Proposition~\ref{easytop}.
\end{proof}

A few additional results of this flavour are found in \cite[\S 25]{S06}.\\

Fix an LCA group $L = L_1\oplus\R^n$, where $L_1$ contains a compact-open subgroup. 
Corollaries~\ref{autinverse} and~\ref{lcadet} imply that 
\begin{align}
\label{eq:AutL}
\Aut(L) \iso \begin{pmatrix}\Aut(L_1) & \Hom(\R^n,L_1)\\\Hom(L_1,\R^n) & \GL_n(\R)\end{pmatrix}
\end{align}
as (abstract) groups. Theorem \ref{AutLdecomp} is established once we show that this isomorphism is topological, a result that follows from the equivalence of (i) and (iv) in Theorem \ref{AutLequivs} below.

\begin{thm}\label{AutLequivs}
Let $(\phi_\lambda)$ be a net in $\Aut(L)$, and $\phi\in\Aut(L)$. The following statements are equivalent:

	\begin{enumerate}[label=\rm({\roman*})]
		\item $\phi_\lambda \tog \phi$ in $\Aut(L)$;
		
		\item $\phi_\lambda \toCO \phi$ and $\det(\phi_\lambda)\tog\det(\phi)$;
		
		\item $\phi_\lambda \toCO \phi$ and $(\det(\phi_\lambda))^{-1} \toCO (\det(\phi))^{-1}$;
		
		\item $\phi_\lambda \toCO \phi$ and $\alpha_\lambda\tog\alpha$;
		
		\item $\phi_\lambda \toCO \phi$ and $\alpha_\lambda^{-1} \toCO \alpha^{-1}$.
	\end{enumerate}
\end{thm}

\begin{proof}
Throughout the proof, we identify automorphisms in $\Aut(L)$ with their 
matrix representations as provided in (\ref{eq:AutL}), and use the 
following convention to denote components:
$\phi_\lambda = \begin{pmatrix}\alpha_\lambda&\beta_\lambda\\\gamma_\lambda&\delta_\lambda\end{pmatrix}$ and $\phi= \begin{pmatrix}\alpha&\beta\\\gamma&\delta\end{pmatrix}$. Furthermore, by Remark \ref{endRn}, the compact-open and $g$-topologies coincide on $\Aut(\R^n)$. So $\delta_\lambda\to\delta$ if and only if $\delta_\lambda^{-1}\to\delta^{-1}$, and given one, we need not verify the other.

(i)$\Longrightarrow$(ii): 
Since the $g$-topology is finer than the 
compact-open one, it follows that $\phi_\lambda\toCO\phi$. By 
Proposition~\ref{gconv} applied to $\Aut(L_1)$, it suffices to show that 
$\det(\phi_\lambda)\toCO\det(A)$ and \\
$\det(\phi_\lambda)^{-1}\toCO\det(\phi)^{-1}$. Since $\alpha_\lambda\toCO\alpha$, $\beta_\lambda\to\beta$, $\gamma_\lambda\to\gamma$ and $\delta_\lambda\to\delta$, where each of the spaces involved carries the compact-open topology, it follows by Proposition \ref{cocomp} that
	\[
		\det(\phi_\lambda) = \alpha_\lambda-\beta_\lambda\delta_\lambda^{-1}\gamma_\lambda \toCO \alpha-\beta\delta^{-1}\gamma = \det(\phi).
	\]
	By Corollary \ref{autinverse}, 
	\begin{align*}
		\phi_\lambda^{-1} &= \begin{pmatrix}
					(\det(\phi_\lambda))^{-1} &  -(\det(\phi_\lambda))^{-1}(\beta_\lambda\delta_\lambda^{-1})\\
					-\delta_\lambda^{-1}\gamma_\lambda(\det(\phi_\lambda))^{-1} & \delta_\lambda^{-1}
				\end{pmatrix}, \text{ and }\\
		\phi^{-1} &= \begin{pmatrix}
					(\det(\phi))^{-1} &  -(\det(\phi))^{-1}(\beta\delta^{-1})\\
					-\delta^{-1}\gamma(\det(\phi))^{-1} & \delta^{-1}
				\end{pmatrix}.
	\end{align*}
Since $\phi_\lambda^{-1}\toCO\phi^{-1}$, in particular, the $(1,1)$-entry 
of $\varphi_\lambda^{-1}$ converges to the $(1,1)$-entry of 
$\varphi^{-1}$. Hence, $(\det(\phi_\lambda))^{-1}\toCO(\det(\phi))^{-1}$.

	(ii)$\Longrightarrow$(iii) follows from  Proposition \ref{gconv} 
applied to $\Aut(L_1)$.
	
	(iii)$\Longrightarrow$(i): Since $\phi_\lambda\toCO\phi$, by Proposition \ref{gconv}, it suffices to show that $\phi_\lambda^{-1}\toCO\phi^{-1}$. We know that $(\det(\phi_\lambda))^{-1} \toCO (\det(\phi))^{-1}$, so
by  Proposition~\ref{cocomp},
	\begin{align*}
		-(\det(\phi_\lambda))^{-1}(\beta_\lambda\delta_\lambda^{-1}) &\to -(\det(\phi))^{-1}(\beta\delta^{-1}) \text{ in $\Hom(\R^n,L_1)$, and}\\
		-\delta_\lambda^{-1}\gamma_\lambda(\det(\phi_\lambda))^{-1} &\to -\delta^{-1}\gamma(\det(\phi))^{-1} \text{ in $\Hom(L_1,R^n)$.}
	\end{align*}
	Thus, one has
	\[
		\begin{pmatrix}
					(\det(\phi_\lambda))^{-1} &  -(\det(\phi_\lambda))^{-1}(\beta_\lambda\delta_\lambda^{-1})\\
					-\delta_\lambda^{-1}\gamma_\lambda(\det(\phi_\lambda))^{-1} & \delta_\lambda^{-1}
				\end{pmatrix} \toCO
		\begin{pmatrix}
					(\det(\phi))^{-1} &  -(\det(\phi))^{-1}(\beta\delta^{-1})\\
					-\delta^{-1}\gamma(\det(\phi))^{-1} & \delta^{-1}
				\end{pmatrix}.
	\]
	That is, $\phi_\lambda^{-1}\toCO\phi^{-1}$, and hence, $\phi_\lambda\tog\phi$.
	
	(iii)$\Longrightarrow$(iv): Since $\phi_\lambda\toCO\phi$,
one has $\alpha_\lambda\toCO\alpha$. Thus, by
Proposition~\ref{gconv}, it suffices  to show that 
$\alpha_\lambda^{-1}\to\alpha^{-1}$. By (\ref{alphainv}),
	\begin{align*}
		\alpha_\lambda^{-1} &= (\det(\phi_\lambda))^{-1}(\1_{B}-\beta_\lambda\delta_\lambda^{-1}\gamma_\lambda(\det(\phi_\lambda))^{-1})\\
		\alpha^{-1} &= (\det(\phi))^{-1}(\1_{B}-\beta\delta^{-1}\gamma(\det(\phi))^{-1}).
	\end{align*}
Therefore, by Proposition~\ref{cocomp}, 
	\[
		\alpha_\lambda^{-1} = (\det(\phi_\lambda))^{-1}(\1_{B}-\beta_\lambda\delta_\lambda^{-1}\gamma_\lambda(\det(\phi_\lambda))^{-1}) \toCO  (\det(\phi))^{-1}(\1_{B}-\beta\delta^{-1}\gamma(\det(\phi))^{-1}) = \alpha^{-1}.
	\]
	
(iv)$\Longrightarrow$(v) follows from Proposition~\ref{gconv} applied to 
$\Aut(L_1)$.
	
	(v)$\Longrightarrow$(iii): By~(\ref{detinv}),
	\begin{align*}
		(\det(\phi_\lambda))^{-1} &= \alpha_\lambda^{-1}(\1_{B} + \beta_\lambda\delta_\lambda^{-1}\gamma_\lambda\alpha_\lambda^{-1})\\
		(\det(\phi))^{-1} &= \alpha^{-1}(\1_{B} + \beta\delta^{-1}\gamma\alpha^{-1})
	\end{align*}
	Therefore, by Proposition~\ref{cocomp},
	\[
	(\det(\phi_\lambda))^{-1} = \alpha_\lambda^{-1}(\1_{B} + \beta_\lambda\delta_\lambda^{-1}\gamma_\lambda\alpha_\lambda^{-1})\toCO\alpha^{-1}(\1_{B} + \beta\delta^{-1}\gamma\alpha^{-1}) = (\det(\phi))^{-1}.
	\]
	This establishes the remaining equivalence.
\end{proof}

We remark that the previous theorem has a striking similarity to the purely algebraic Theorem~\ref{catthm}. In both cases, we have reduced a question regarding elements of 
$\Aut(L)$ to a question regarding only its diagonal components, utilizing the quasi-determinant as an intermediate step. Also, observe that (i)$\Longrightarrow$(ii) in Theorem \ref{AutLequivs} implies that the quasi-determinant $\det:\Aut(L)\to\Aut(L_1)$ is continuous.

\section*{Acknowledgments}

This research was conducted as part of an NSERC Undergraduate Summer Research Award under the supervision of G\'{a}bor Luk\'{a}cs. I thank Dr. Luk\'{a}cs for his wisdom, guidance, attention to detail, and understanding; without him, this work would simply have not been possible.

I would also like to thank Karen Kipper for carefully proof-reading this paper for grammar and spelling.

\end{document}